%% file: max_spectral_gaps_v3.tex
\documentclass[11pt]{article}
\usepackage{tikz}
\usepackage{graphicx}
\usepackage{amsthm, amsmath, amssymb, tikz, bm}
\usepackage{mathtools}
\usepackage{xifthen}
\usepackage[font=small]{caption}
\usepackage{comment}
\usepackage{tablefootnote}
\usepackage{array}
\usepackage{diagbox}
\usepackage[colorlinks=true,
linkcolor=blue,citecolor=blue,
urlcolor=blue]{hyperref}
\usepackage[margin=0.8in]{geometry} 

\usepackage[shortlabels]{enumitem}
\usepackage{todonotes}
\usepackage{stackengine}
\newcommand\xrowht[2][0]{\addstackgap[.5\dimexpr#2\relax]{\vphantom{#1}}}

\usetikzlibrary{calc,shapes, backgrounds}
\allowdisplaybreaks

\newtheorem{lemma}{Lemma}
\newtheorem{theorem}{Theorem}

\newtheorem{remark}{Remark}
\newtheorem{conjecture}{Conjecture}

\newcommand{\tr}{{\rm tr}}

\newcolumntype{L}[1]{>{\raggedright\arraybackslash}p{#1}}
\newcolumntype{C}[1]{>{\centering\arraybackslash}p{#1}}
\newcolumntype{R}[1]{>{\raggedleft\arraybackslash}p{#1}}
\newcommand{\mt}[2]{\begin{tabular}{@{}L{1cm}R{1cm}@{}}
&#2\\
#1&
\end{tabular}}

\title{Maximum spectral gaps of graphs}

\author{George Brooks
\thanks{University of South Carolina, Columbia, SC. ({\tt ghbrooks@email.sc.edu}). The author is partially supported by NSF DMS 2038080 grant.}
\and
William Linz \thanks{University of South Carolina, Columbia, SC. ({\tt wlinz@mailbox.sc.edu}). The author is partially supported by NSF DMS 2038080 grant.}
\and Linyuan Lu \thanks{University of South Carolina, Columbia, SC. ({\tt lu@math.sc.edu}). The author is partially supported by NSF DMS 2038080 grant.}
}

\begin{document}
\newcommand{\spread}{{\rm spread}}
\newcommand{\G}{{\cal G}}
\maketitle

\abstract{
The \emph{spread} of a graph $G$ is the difference $\lambda_1 - \lambda_n$ between the largest and smallest eigenvalues of its adjacency matrix. Breen, Riasanovsky, Tait and Urschel recently determined the graph on $n$ vertices with maximum spread for sufficiently large $n$. In this paper, we study a related question of maximizing the difference $\lambda_{i+1} - \lambda_{n-j}$ for a given pair $(i, j)$ over all graphs on $n$ vertices. We give upper bounds for all pairs $(i, j)$, exhibit an infinite family of pairs where the bound is tight, and show that for the pair $(1, 0)$ the extremal example is unique. These results contribute to a line of inquiry pioneered by Nikiforov aiming to maximize different linear combinations of eigenvalues over all graphs on $n$ vertices.
}

\section{Introduction}
Given a graph $G$, the \emph{spread} of $G$ is the difference between the largest eigenvalue and the smallest eigenvalue of its adjacency matrix $A$. The spread of graphs was introduced and first studied by Gregory, Hershkowitz and Kirkland~\cite{GHK2001} as a specialization of the concept of the spread of a matrix introduced by Mirsky~\cite{Mir}. In 2021, answering a question in \cite{GHK2001}, Breen, Riasanovsky, Tait, and Urschel~\cite{BRTU2021+} determined the maximum of the spread among all simple graphs on $n$ vertices. For all simple graphs $G$ on $n$ vertices, they show that the spread is at most $\frac{2}{\sqrt{3}}n$, where the maximum spread is (asymptotically) achieved by a split graph, which is obtained by joining a clique of order $\left\lfloor\frac{2n}{3}\right\rfloor$ 
and an independent set of order $\left \lceil\frac{n}{3} \right\rceil$.

In this paper, we will consider several variations of the maximum spread problem.  Let $\lambda_1 \ge \lambda_2 \ge \ldots \ge \lambda_n$ be the eigenvalues of $A$ sorted in nonincreasing order.
For any fixed integers $i, j\geq 0$, we call the difference $\lambda_{i+1} -\lambda_{n-j}$,
the \emph{$(i,j)$-spread} of $G$, denoted by $\spread_{i,j}(G)$.
Let $S_n$ denote the set of all simple graphs on $n$ vertices.
We define 
\[\spread_{i,j}(n)=\max\{\spread_{i,j}(G) \colon G\in S_n \}\]
and
\[s_{i,j}=\lim_{n\to\infty} \frac{\spread_{i, j}(n)}{n}.\]

 Nikiforov~\cite{Nik06} proved that this limit always exists. In fact, Nikiforov proved that
\[s_{i,j} = \sup_{n} \frac{\spread_{i,j}(n)}{n}.\] 

In our notation, Breen-Riasanovsky-Tait-Urschel's result can be stated as
\[s_{0,0}=\frac{2}{\sqrt{3}}.\]

To the best of our knowledge, the quantity $s_{i, j}$ has not been previously studied except in special cases. The general question of maximizing a linear combination of eigenvalues of a graph was stated by Nikiforov~\cite[Theorem 1]{Nik06}, who showed that if $F(G)$ is a linear combination of the eigenvalues of a graph $G$ in the form $\sum_{i=1}^k\alpha_i\lambda_i + \beta_i\lambda_{n-i+1}$, then $\lim_{n\rightarrow \infty}\frac{F(G)}{n}$ exists. Specific cases have been more extensively studied; for example, Ebrahimi, Mohar, Nikiforov and Ahmady~\cite{EMNA08} provided good upper and lower bounds on $\lambda_1$ + $\lambda_2$.

In this paper, we provide bounds on $s_{i, j}$ for all $i$ and $j$ and find infinitely many pairs for which the bounds are tight. Our key idea is to focus on graphs with possible self-loops. Let $L_n$ denote the set of all graphs on $n$ vertices with at most one possible self-loop at each vertex. This is a slightly larger class of graphs than all simple graphs. From the matrix-theoretic viewpoint, instead of considering only symmetric $n\times n$ $(0, 1)$-matrices with all $0$s on the diagonal (this family of matrices corresponds to $S_n$), we consider the problems over the family of all symmetric $n\times n$ $(0, 1)$-matrices. Let 
\[\spread^*_{i, j}(n) = \max\{\spread_{i, j}(G): G\in L_n\}\]
and 
\[s_{i,j}^* = \lim_{n\rightarrow \infty}\frac{\spread^*_{i, j}(n)}{n}.\]
We show in Section 2 that $s_{i, j}^*$ exists and $s_{i, j}^* = s_{i, j}$. From now on, we use the notation $G^*$ to denote a graph with at most one self-loop per vertex whose underlying simple graph (obtained after deleting the loops) is the graph $G$. Naturally, there are many nonisomorphic graphs with at most one loop per vertex which have the same underlying simple graph $G$. Whenever it is necessary to distinguish which vertices have self-loops (for example, in Theorem~\ref{t1} and in Section 4), we will indicate this explicitly using an open dot to represent a loopless vertex and a closed dot for a vertex with a self-loop.

We have the following general upper bounds for $s_{i,j}$. The bounds are not tight in general.
\begin{theorem}\label{thm:genijub}
    For any positive integer $i$ and nonnegative integer $j$ with $i, j \le n$, for all graphs $G^*$ on $n$ vertices with at most one loop per vertex, we have
    \[ \lambda_{i+1}(G^*)-\lambda_{n-j}(G^*)\leq \frac{n}{2}\sqrt{\frac{i+j+1}{i(j+1)}}. \]
    Thus, 
    \[s_{i,j} \le \frac{1}{2}\sqrt{\frac{i+j+1}{i(j+1)}}.\]
\end{theorem}

\begin{remark}\label{rem:equality}
    Equality holds in Theorem~\ref{thm:genijub} if and only if $\lambda_1 = d = \frac{n}{2}$, $\lambda_2 = \cdots = \lambda_{i+1} = \frac{n}{2}\sqrt{\frac{j+1}{i(i+j+1)}}$, $\lambda_{n-j} = \cdots = \lambda_{n} = -\frac{n}{2}\sqrt{\frac{i}{(j+1)(i+j+1)}}$ and $\lambda_{i+2} = \cdots = \lambda_{n-j-1} = 0$.
\end{remark}

Note that Theorem~\ref{thm:genijub} does not cover linear combinations of eigenvalues of the form $\lambda_1 - \lambda_{n-j}$. The next theorem addresses these cases. 

\begin{theorem}\label{thm:0jthm}
    For any integer $0\leq j\leq n-1$, for all graphs $G^*$ on $n$ vertices with at most one loop per vertex, we have
    \[ \lambda_{1}(G^*)-\lambda_{n-j}(G^*)\leq \frac{n}{2}\left(1+\sqrt{\frac{j+2}{j+1}}\right). \]
    Thus,
    \[s_{0, j} \le \frac{1}{2}\left(1+\sqrt{\frac{j+2}{j+1}}\right)\]
\end{theorem}

The upper bound in Theorem~\ref{thm:genijub} can be attained in some cases. We show that the bound on $s_{k, k-1}$  is tight for infinitely many $k$. 

\begin{theorem}\label{thm:allk}
For all graphs $G^*$ on $n$ vertices with at most one loop per vertex, we have
\[\lambda_{k+1}(G^*) - \lambda_{n-k+1}(G^*) \le \frac{n}{\sqrt{2k}}.\]
Equality holds if there is a symmetric Hadamard matrix of order $2k$. In particular, if such a symmetric Hadamard matrix of order $2k$ exists, then we have $s_{k, k-1} = \frac{1}{\sqrt{2k}}$. 
\end{theorem}

If $k = 1$, then we can additionally determine the unique graphs with loops which attain the upper bound of Theorem~\ref{thm:allk}. Recall that the \emph{$t$-blowup} of a graph $G^*$ is the graph $G^*(t)$ where the vertices of $G^*$ are replaced with independent sets of size $t$ and the edges are replaced with complete bipartite graphs $K_{t, t}$; a more formal definition of graph blowups is given in Section 2. 

\begin{theorem}\label{t1}
Among all graphs $G^*$ with possible self-loops on $n$ vertices, we have
\[\lambda_2(G^*)-\lambda_n(G^*)\leq \frac{n}{\sqrt{2}}.\]
Equality holds if and only if $G^*$ is a $t$-blowup of the closed path $P^*_4$ in Figure~\ref{fig:P4*}.

\begin{figure}[ht]
    \centering
    \input{figures/closedP4.tikz}
    \caption{Closed path $P^*_4$.}
    \label{fig:P4*}
\end{figure}
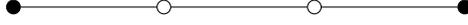
In particular, we have $s_{1,0}=\frac{1}{\sqrt{2}}$.
\end{theorem}

\begin{remark}
The eigenvalues of $P^*_4$ are $2, \sqrt{2},0,-\sqrt{2}$. The eigenvalues
of the blowup $P^*_4(t)$ are $2t$, $\sqrt{2}t$,$0,\ldots, 0$, and $-\sqrt{2}t$. Thus,
we have,
\[\lambda_2(P^*_4(t))-\lambda_n(P^*_4(t))=2\sqrt{2}t=\frac{n}{\sqrt{2}}.\]
\end{remark}

The proofs of Theorems~\ref{thm:genijub}, \ref{thm:0jthm}, \ref{thm:allk} and \ref{t1} are given in Section 3. In Section 4, we list lower bounds for $s_{i, j}$ for many different pairs $(i, j)$ and make a number of conjectures for precise values of $s_{i, j}$. 

\section{Notations and Lemmas}
Let $G^*$ be a graph (possibly with at most one loop per vertex) with vertex set $V(G^*) = \{v_1, \ldots, v_n\}$ and edge set $E$. Let $A = (a_{ij})$ be the adjacency matrix of $G^*$, which is defined so that $a_{ij} = 1$ if $ij \in E(G^*)$ and $a_{ij} = 0$ if $ij \notin E(G^*)$. In particular, if there is a loop at vertex $i$ in the graph $G^*$, then $a_{ii} = 1$.  Note that $A$ is a real-symmetric matrix with each entry $0$ or $1$, so its eigenvalues are all real, and we denote the eigenvalues in nonincreasing order as $\lambda_1 \ge \lambda_2 \ge \cdots \ge \lambda_n$. 

The \emph{degree} of vertex $i$ in the graph $G^*$ is $d_i = \sum_{j=1}^na_{ij}$, and the \emph{average degree} of the graph $G^*$ is $d =\frac{\sum_{i=1}^nd_i}{n}$. Under these conventions, it no longer holds that the sum of the degrees of $G^*$ is equal to twice the number of edges, as is the case for simple loopfree graphs. However, it still holds that the spectral radius of $G^*$ satisfies the inequality $\lambda_1 \ge d$ and that the sum of the squares of the eigenvalues of $G^*$ satisfies $\lambda_1^2 + \ldots + \lambda_n^2 = \text{Tr}(A^2) = nd$.

 For a positive integer $t$, let $G^*(t)$ be the graph with vertex set $V(G^*(t)) = \{ v_{i, j}: 1\le i\le n, 1\le j\le t\}$ and edge set $E(G^*(t)) = \{v_{i_1, j_1}v_{i_2,j_2}: i_1i_2 \in E\}$. The graph $G^*(t)$ is called the \emph{$t$-blowup} of the graph $G^*$. 

\begin{lemma}[Eigenvalues of the blowup graph]\label{lem:eigsblowup}
Suppose the eigenvalues of $G^*$ are $\lambda_1 \ge \cdots \ge \lambda_n$. Then, the eigenvalues of $G^*(t)$ are $\lambda_1t, \lambda_2t, \ldots, \lambda_nt$ with $(t-1)n$ additional $0$s. 
\end{lemma}

\begin{proof}
    The adjacency matrix of $G^*(t)$ is $A(G^*(t)) = A \otimes J_t$, where $\otimes$ is the Kronecker product of matrices and $J_t$ is the $t\times t$ all ones matrix. The lemma immediately follows from standard facts about the spectrum of the Kronecker product. 
 \end{proof}

\begin{lemma}\label{lem:simeigs}
Let $G^*$ be a graph on $n$ vertices with at most one self-loop at each vertex. Let $G$ be the underlying simple graph obtained by removing all the self-loops from $G^*$ (and keeping all of the vertices). Then, for $1\le k\le n$, $\lambda_k(G^*) - 1 \leq \lambda_k(G) \leq \lambda_k(G^*)$. 
\end{lemma}

\begin{proof}
    Let $A$ be the adjacency matrix of $G^*$ and let $B$ be the adjacency matrix of $G$. Then, $A = B + D$, where $D$ is a diagonal $(0, 1)$-matrix. Since $D$ is positive semidefinite, by the Rayleigh quotient $\lambda_k(G) \le \lambda_k(G^*)$. On the other hand, $\lambda_1(D) \le 1$, so by Weyl's inequality,
    \[\lambda_k(G^*) - 1 \leq \lambda_k(G^*) - \lambda_1(D) \leq \lambda_k(G).\]
\end{proof}

Lemma~\ref{lem:simeigs} shows that in the limit there is no loss of information in working with the set of graphs $L_n$ with at most one self-loop per vertex instead of the set of simple graphs $S_n$. 

\begin{lemma}\label{lem:sij=sij*}
The limit $s_{i,j}^*$ exists and 
\[s_{i,j}^*=s_{i,j}.\]
\end{lemma}

\begin{proof}

If $G^*$ is a graph on $n$ vertices with at most one loop per vertex, and $G$ is the underlying graph obtained by removing the loops from $G^*$, then Lemma~\ref{lem:simeigs} gives

\[\frac{\lambda_{i+1}(G) - \lambda_{n-j}(G)}{n} - \frac{1}{n} \le \frac{\lambda_{i+1}(G^*) - \lambda_{n-j}(G^*)}{n} \le \frac{\lambda_{i+1}(G) - \lambda_{n-j}(G)}{n} + \frac{1}{n}.\]

By taking limits as $n\rightarrow \infty$, the squeeze theorem shows $s_{i, j}^*$ exists and $s_{i, j}^* = s_{i, j}$. 
\end{proof}

\section{Proofs of the theorems}
We first give the proofs of Theorems \ref{thm:genijub} and \ref{thm:0jthm}. The upper bounds follow from applications of the Cauchy-Schwarz inequality. 

\begin{proof}[Proof of Theorem~\ref{thm:genijub}]
    Assume $\lambda_{i+1} \ge 0 \ge \lambda_{n-j}$. Then, by the Cauchy-Schwarz inequality, we have 
    \[(\lambda_{i+1} - \lambda_{n-j})^2 = (\lambda_{i+1} + |\lambda_{n-j}|)^2 \le \left(\frac{1}{i} + \frac{1}{j+1}\right)(i\lambda_{i+1}^2 + (j+1)\lambda_{n-j}^2)\]
    \[\implies  \lambda_{i+1} - \lambda_{n-j}  \le \sqrt{\left(\frac{1}{i}+\frac{1}{j+1}\right)(i\lambda_{i+1}^2 + (j+1)\lambda_{n-j}^2)}.\]

    We have $i\lambda_{i+1}^2 \le \lambda_2^2 + \ldots + \lambda_{i+1}^2$ and $(j+1)\lambda_{n-j}^2 \le \lambda_{n-j}^2 + \cdots + \lambda_n^2$. Since $\lambda_1^2 + \cdots + \lambda_n^2 = nd$ and $\lambda_1 \ge d$, it follows that 
    \begin{align*}
        i\lambda_{i+1}^2 + (j+1)\lambda_{n-j}^2 & \le \lambda_2^2 + \cdots + \lambda_{i+1}^2 + \lambda_{n-j}^2 + \cdots + \lambda_n^2\\
        & \le (\lambda_1^2 + \cdots + \lambda_n^2) - \lambda_1^2\\
        & = nd - \lambda_1^2\\
        & \le nd - d^2\\
        &\le \frac{n^2}{4},
    \end{align*}
    where the last inequality follows from the fact that $nd-d^2$ is maximized at $d=n/2$. 
    Hence, 
    \begin{align*}
        \lambda_{i+1} - \lambda_{n-j} & \le \sqrt{\left(\frac{1}{i}+\frac{1}{j+1}\right)(i\lambda_{i+1}^2 + (j+1)\lambda_{n-j}^2)}\\
        & \le \sqrt{\left(\frac{i+j+1}{i(j+1)}\right)\left(\frac{n^2}{4}\right)}\\
        &= \frac{n}{2}\sqrt{\frac{i+j+1}{i(j+1)}}.
    \end{align*}

     Equality holds in Theorem~\ref{thm:genijub} if and only if it holds at each step in the argument, so in particular equality holds only if $\lambda_1 = d = \frac{n}{2}$, $\lambda_2 = \cdots = \lambda_{i+1} = \frac{n}{2}\sqrt{\frac{j+1}{i(i+j+1)}}$, $\lambda_{n-j} = \cdots = \lambda_{n} = -\frac{n}{2}\sqrt{\frac{i}{(j+1)(i+j+1)}}$ and $\lambda_{i+2} = \cdots = \lambda_{n-j-1} = 0$.

    If $\lambda_{i+1}$ and $\lambda_{n-j}$ have the same sign, or if $\lambda_{i+1} \le 0 \le \lambda_{n-j}$, then $\lambda_{i+1} - \lambda_{n-j} \le \max\{|\lambda_{i+1}|, |\lambda_{n-j}|\}$. We have $|\lambda_{i+1}| \le \lambda_{i+1}^*$ and $|\lambda_{n-j}| \le \lambda_{j+2}^*$, where $\lambda_{i+1}^*$ is the $(i+1)$th largest singular value of $G^*$, so (as proved by Nikiforov~\cite{Nik15}),
    \[\lambda_{i+1}^*\le \frac{n}{2\sqrt{i}}\]
    and
    \[\lambda_{j+2}^*\le \frac{n}{2\sqrt{j+1}}.\]
    In either case, the bound is less than $\frac{n}{2}\sqrt{\frac{i+j+1}{i(j+1)}}$, so equality cannot hold.  

\end{proof}

We now give the proof of Theorem~\ref{thm:0jthm}.
\begin{proof}[Proof of Theorem~\ref{thm:0jthm}]
If $\lambda_{n-j}(G^*) \ge 0$, then $\lambda_1(G^*) - \lambda_{n-j}(G^*)\le \lambda_1(G^*) \le n < \frac{n}{2}\left(1+\sqrt{\frac{j+2}{j+1}}\right)$.

Assume now that $\lambda_{n-j}(G^*) < 0$. 
Let $z>1$ be a real number chosen later.  Then, by the Cauchy-Schwarz inequality, we have 
\[(\lambda_1 - \lambda_{n-j})^2 = (\lambda_1 + |\lambda_{n-j}|)^2 \le \left(1+\frac{1}{z(j+1)}\right)(\lambda_1^2 + z(j+1)\lambda_{n-j}^2) \]
\[ \implies \lambda_1 - \lambda_{n-j}    \le \sqrt{\left(1+\frac{1}{z(j+1)}\right)(\lambda_1^2 + z(j+1)\lambda_{n-j}^2)}.\\ \]
Note that $z(j+1)\lambda_{n-j}^2 \le z(\lambda_{n-j}^2 + \ldots +\lambda_n^2)$, so since $\lambda_1^2 + \cdots + \lambda_n^2 = nd$ and $\lambda_1 \ge d$, we have
\begin{align*}
    \lambda_1^2 + z(j+1)\lambda_{n-j}^2 & \le \lambda_1^2 + z(\lambda_{n-j}^2 + \ldots + \lambda_n^2)\\
    & \le z(\lambda_1^2 + \cdots + \lambda_n^2) - (z-1)\lambda_1^2\\
    & = znd - (z-1)\lambda_1^2\\
    & \le znd - (z-1)d^2\\
    & \le \frac{z^2n^2}{4(z-1)},
\end{align*}
where the last inequality follows from the fact that $znd - (z-1)d^2$ is maximized at $d = zn/(2(z-1))$. Therefore, we obtain 
\begin{align*}
    \lambda_1 - \lambda_{n-j}   & \le  \sqrt{ \left(1+\frac{1}{z(j+1)}\right) \frac{z^2n^2}{4(z-1)}}\\
    &=\frac{n}{2}\sqrt{ \left(1+\frac{1}{z(j+1)}\right) \frac{z^2}{(z-1)}}.
\end{align*}
Let $f(z)= \sqrt{\left(1+\frac{1}{z(j+1)}\right) \frac{z^2}{(z-1)}}$. It turns out $f(z)$ reaches the minimum 
value $1+\sqrt{\frac{j+2}{j+1}}$
 at  $z=1+\sqrt{\frac{j+2}{j+1}}$. 
 By setting $z=1+\sqrt{\frac{j+2}{j+1}}$ and simplifying, we have
\[ \lambda_1 - \lambda_{n-j}\leq 
\frac{n}{2}\, f\left (1+\sqrt{\frac{j+2}{j+1}}\right )=\frac{n}{2}\left(1 + \sqrt{\frac{j+2}{j+1}}\right).\]

\end{proof}

Let $i=k$ and $j=k-1$. Then, by Theorem~\ref{thm:genijub}, we have
\[\lambda_{k+1}(G^*) - \lambda_{n-k+1}(G^*) \le \frac{n}{\sqrt{2k}},\]
for all graphs on $n$ vertices with at most one loop per vertex. We can show tightness of the bound for infinitely many $k$. Suppose that there is a symmetric Hadamard matrix $H$ of order $2k$. Let $K = \begin{pmatrix}1 & -1\\ -1 & 1\end{pmatrix}$. 
Then, the graph $G^*$ on $4k$ vertices whose adjacency matrix $A$ is the  $(0, 1)$-symmetric matrix given by 
\[A = \frac12(K\otimes H + J_{4k})\]
has the appropriate eigenvalues needed for equality in Theorem~\ref{thm:genijub}. This proves Theorem~\ref{thm:allk}. (Nikiforov~\cite{Nik15} first used matrices of this form to give tight or near-tight constructions for related problems such as maximizing the $k$th singular value of a graph on $n$ vertices). 

When $k=1$, using the Hadamard matrix \[H_2 = \begin{pmatrix}1 & 1\\ 1 & -1\end{pmatrix},\] the graph with adjacency matrix $\frac12(K\otimes H_2 + J_4)$ is $P_4^*$. When $k=2$, using the unique Hadamard matrix $H_4$ of order $4$, the graph with adjacency matrix $\frac12(K\otimes H_4 + J_8)$ is the closed cube $Q_3^*$ shown in Figure~\ref{fig:closed-cube}. 

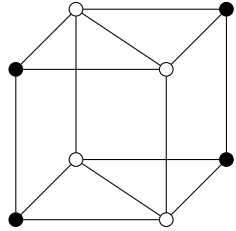
\begin{figure}[ht]
   \centering
    \input{figures/closedCube.tikz}
    \caption{Closed cube $Q^*_3$.}
    \label{fig:closed-cube}
\end{figure}

We now give the proof of Theorem~\ref{t1}, which further extends Theorem~\ref{thm:allk} by showing that when $k=1$, the blowups of the graph $P_4^*$ obtained from the symmetric Hadamard matrix of order $2$ are the unique extremal graphs meeting the bound.

\begin{proof}[Proof of Theorem~\ref{t1}]
    By the proof of Theorem~\ref{thm:allk}, equality holds when
    $\lambda_1=d=\frac{n}{2}$, 
    $\lambda_2=\frac{\sqrt{2}}{4}n$,
    $\lambda_3=\cdots=\lambda_{n-1}=0$, and
    $\lambda_n=-\frac{\sqrt{2}}{4}n$.

 In the remainder of the proof, we will show that any graph with the above spectrum must be a blowup of $P^*_4$.  
Since $\lambda_1=d$, $G^*$ must be a regular graph.

Note that $\tr(A)=\sum_{i=1}^n \lambda_i=\frac{n}{2}$. The order $n$ must be even and half of the vertices must have a self-loop while the other half are loopless. Let $S$ denote the set of vertices with a self-loop and let $\bar S$ be the set of vertices without loops.

For $i=1,\dots,n$, let $v_i$ be the orthogonal unit eigenvector corresponding to $\lambda_i$.
In particular, we have $v_1=\frac{1}{\sqrt{n}}{\bf 1}$, where $\bf 1$ is the all ones vector.
We have
\[A=\sum_{i=1}^n \lambda_i v_i v_i'   
= \frac{1}{2} J + \frac{\sqrt{2}}{4}n v_2 v_2'
-\frac{\sqrt{2}}{4}n v_nv_n'.
\]
In particular, we have
\begin{equation}\label{eq:A1}
  A-\frac{1}{2}J=  \frac{\sqrt{2}}{4}n (
v_2 v_2'-v_n v_n'). 
\end{equation}

Squaring both sides, we get
\[A^2-AJ +\frac{1}{4}J^2 = \frac{n^2}{8}
(v_2 v_2'-v_n v_n')^2.\]
Note that $AJ=\frac{n}{2}J$, $J^2=nJ$,
and $(v_2 v_2'-v_n v_n')^2=v_2 v_2'+v_n v_n'$.
We get
\begin{equation}\label{eq:A2}
A^2-\frac{n}{4}J=\frac{n^2}{8}(v_2 v_2'+v_n v_n').
\end{equation}
Combining Equations \eqref{eq:A1} and \eqref{eq:A2}, we have
\begin{equation*}
    v_2v_2'= \frac{4}{n^2}\left(A^2-\frac{n}{4}J\right) + \frac{\sqrt{2}}{n}\left(A-\frac{1}{2}J\right).
\end{equation*}
Taking the values at the $i$th diagonal entry, we have
\begin{align*}
    v_2(i)^2 &= \frac{4}{n^2}\left(\frac{n}{2}-\frac{n}{4}\right) +  \frac{\sqrt{2}}{n} \left({\bf 1}_{S(i)}-\frac{1}{2}\right)\\
    &=\begin{cases}
        \frac{1}{n}\left(1+ \frac{1}{\sqrt{2}}\right) & \mbox{ if } i\in S; \\
         \frac{1}{n}\left(1- \frac{1}{\sqrt{2}}\right) &\mbox{ if } i\in \bar S.
    \end{cases}
\end{align*}
Here ${\bf 1}_{S(i)}$ takes value 1 at $i\in S$ and $0$ otherwise.

Similarly, we have
\[v_n(i)^2
=\begin{cases}
        \frac{1}{n}\left(1- \frac{1}{\sqrt{2}}\right) & \mbox{ if } i\in S; \\
         \frac{1}{n}\left(1+ \frac{1}{\sqrt{2}}\right) &\mbox{ if } i\in \bar S.
    \end{cases}
\]
Let $S^+$ (respectively $\bar S^+)$ be the set of vertices in $S$ (respectively $\bar S$) where $v_2(i)$ takes a positive value while
$S^-$ (respectively $\bar S^-$) be the set of vertices in $S$ (respectively $\bar S$) where $v_2(i)$ takes a negative value. 
Then we have
\[
v_2(i)=
\begin{cases}
        \frac{1}{\sqrt{n}}\sqrt{1+ \frac{1}{\sqrt{2}}} & \mbox{ if } i\in S^+; \\
       \frac{1}{\sqrt{n}}\sqrt{1- \frac{1}{\sqrt{2}}} & \mbox{ if } i\in \bar S^+; \\
         -\frac{1}{\sqrt{n}}\sqrt{1- \frac{1}{\sqrt{2}}} & \mbox{ if } i\in \bar S^-;\\
         
         -\frac{1}{\sqrt{n}}\sqrt{1+ \frac{1}{\sqrt{2}}} & \mbox{ if } i\in S^-.
    \end{cases}
\]
Since $v_2\cdot v_1=0$, we have
\[(|S^+|-|S^-|) \sqrt{1+ \frac{1}{\sqrt{2}}} +
(|\bar S^+|-|\bar S^-|) \sqrt{1- \frac{1}{\sqrt{2}}} =0.
\]
This implies 
\[ |S^+|=|S^-|  \mbox{ and } |\bar S^+|=|\bar S^-|.\]
Since $|S|=|\bar S|=\frac{n}{2}$, we have
\[|S^+| =|\bar S^+|=|\bar S^-|=|S^-|=\frac{n}{4}.\]
Thus $n$ is divisible by $4$.

Since $v_n$ is orthogonal to $v_2$, the signs distribute differently on $S^+$,  $\bar S^+$, $\bar S^-$, and $S^-$. Up to a choice of sign, we must have
\[
v_n(i)=
\begin{cases}
        -\frac{1}{\sqrt{n}}\sqrt{1- \frac{1}{\sqrt{2}}} & \mbox{ if } i\in S^+; \\
       \frac{1}{\sqrt{n}}\sqrt{1+ \frac{1}{\sqrt{2}}} & \mbox{ if } i\in \bar S^+; \\
         -\frac{1}{\sqrt{n}}\sqrt{1+ \frac{1}{\sqrt{2}}} & \mbox{ if } i\in \bar S^-;\\
         \frac{1}{\sqrt{n}}\sqrt{1- \frac{1}{\sqrt{2}}} & \mbox{ if } i\in S^-.
    \end{cases}
\]

Now we write the following matrix as $4\times 4$ block matrices according to the vertex partition
$V(G)=S^+\cup \bar S^+ \cup \bar S^- \cup S^-$. We have
\begin{align*}
    A &= \frac{1}{2} J + \frac{\sqrt{2}}{4}n (v_2 v_2'- v_nv_n')\\
&=
\frac{1}{2} J + \frac{1}{4}
\left(
\begin{array}{cccc}
   (\sqrt{2}+1)J   & J &-J &  -(\sqrt{2}+1)J   \\
    J  & (\sqrt{2}-1)J   &  -(\sqrt{2}-1)J & -J\\
    -J & -(\sqrt{2}-1)J   &  (\sqrt{2}-1)J & J \\
    -(\sqrt{2}+1)J    & -J &J &  (\sqrt{2}+1)J   \\
\end{array}
\right)\\
&\hspace*{9mm} -\frac{1}{4}
\left(
\begin{array}{cccc}
   (\sqrt{2}-1)J    & -J &J &  -(\sqrt{2}-1)J   \\
    -J  & (\sqrt{2}+1)J   &  -(\sqrt{2}+1)J & J\\
    J  & -(\sqrt{2}+1)J   &  (\sqrt{2}+1)J & -J\\
    -(\sqrt{2}-1)J    & J &-J &  (\sqrt{2}-1)J   \\
\end{array}
\right)\\
&=\left(
\begin{array}{cccc}
   J    & J &0 &  0    \\
    J   & 0 &J &  0     \\
      0   & J &0 &  J   \\
  0   &  0 &J &  J     \\
\end{array}
\right).
\end{align*}
Therefore, $G^*$ is a blowup of $P^*_4$.
\end{proof}

\section{Lower bounds and conjectures}
In this section, we provide several lower bounds on $s_{i, j}$ and make conjectures regarding the extremal graphs for $s_{i, j}$ for certain pairs $(i, j)$. 

For $i=1$, consider the graph $K_{j+2}\cup K^*_{j+1}$. This graph is the disjoint union of the complete graph $K_{j+2}$ and the complete graph $K_{j+1}$ with loops on all of the vertices of the complete graph $K_{j+1}$. The construction for $j=3$ is shown in Figure~\ref{fig:conj1}. 

\begin{figure}[ht]
    \centering
    \input{figures/conj1.tikz}
    \caption{The graph $K_{5}\cup K_{4}^*$.}
    \label{fig:conj1}
\end{figure}

The eigenvalues of $K_{j+2}\cup K^*_{j+1}$ are 
\[(j+1)^2, 0^{j}, (-1)^{j+1},\]
where, for example, the notation $0^j$ means that the eigenvalue $0$ has multiplicity $j$. Hence, we have $\lambda_2 - \lambda_{n-j} = j+2$, so we obtain the lower bound 
\[s_{1, j} \ge \frac{j+2}{2j+3}.\]  

We conjecture that the blowups of $K_{j+2} \cup K^*_{j+1}$ are extremal for $s_{1, j}$.  

\begin{conjecture}\label{con:1jcon}
For any $j\geq 1$, we have
\[ s_{1,j}=\frac{j+2}{2j+3}.\]
If the conjecture holds, the lower bound is achieved by any graph which is a $t$-blowup of  $K_{j+2}\cup K_{j+1}^*$.
\end{conjecture}

As some evidence for Conjecture~\ref{con:1jcon}, note that
\[\frac{j+2}{2j+3} = \frac{1}{2}+\frac{1}{4j}-\frac{3}{8j^2} + \dotsm, \]
while 
\[  \frac{1}{2}\sqrt{\frac{j+2}{j+1}}=
\frac{1}{2}+\frac{1}{4j}-\frac{5}{16j^2} + \dotsm .\]
Hence, for large $j$ the conjectured exact value for $s_{1, j}$ is close to the upper bound for $s_{1, j}$ obtained from Theorem~\ref{thm:genijub}. 

For $j=0$, consider the blowup of the one-side-closed bipartite graph $K_{(i+1)*,i}$. This graph is the complete bipartite graph $K_{i+1, i}$ with loops added on the vertices in the partite set of size $i+1$. The construction for $i=4$ is shown in Figure~\ref{fig:conj2}. 

\begin{figure}[ht]
    \centering
    \input{figures/conj2.tikz}
    \caption{The graph $K_{(5)*,4}$.}
    \label{fig:conj2}
\end{figure}

The graph $K_{(i+1)*, i}$ has eigenvalues
\[i+1, 1^{i+1}, 0^{i}, (-1)^{i+1}, -i,\]
from which we obtain the lower bound 
\[ s_{i, 0} \ge \frac{i+1}{2i+1}.\]

We conjecture that this lower bound is tight. 

\begin{conjecture}\label{con:i0con}
For any $i\geq 1$, we have
\[ s_{i,0}=\frac{i+1}{2i+1}.\]
If the conjecture holds, the lower bound is achieved by any graph which is a $t$-blowup of  $K_{(i+1)*,i}$.
\end{conjecture}

As some evidence for Conjecture~\ref{con:i0con}, note that 
\[ s_{i, 0} \ge \frac{i+1}{2i+1} =  \frac{1}{2}+\frac{1}{4i}-\frac{1}{8i^2} + \dotsm,\]
while Theorem~\ref{thm:genijub} gives the upper bound
\[s_{i,0} \leq \frac{1}{2}\sqrt{\frac{i+1}{i}}=\frac{1}{2}+\frac{1}{4i}-\frac{1}{16i^2} + \dotsm.\]
The upper bound and lower bound for $s_{i, 0}$ are close for large $i$. 

Let $K_n^{t*}$ be the graph obtained from the complete graph $K_n$ by adding loops on $t$ vertices. The graph $K_8^{4*}$ is shown in Figure~\ref{fig:conj3}.

\begin{figure}[ht]
    \centering
    \input{figures/conj3.tikz}
    \caption{The graph $K_{8}^{4*}$.}
    \label{fig:conj3}
\end{figure}

The graph $K_{2j+4}^{(j+2)*}$ has eigenvalues
\[\frac{2j+3 + \sqrt{4j^2+16j+17}}{2}, j+1, 0^{j+1}, (-1)^j, \frac{2j+3 - \sqrt{4j^2+16j+17}}{2},\] and so $K_{2j+4}^{(j+2)*}$  gives the following lower bound on $s_{0, j}$:
\[s_{0,j}\geq \frac{(2j+5)+\sqrt{4j^2+16j+17}}{4(j+2)}. 
\]

We conjecture that this lower bound is tight for $j\ge 1$. 

\begin{conjecture}\label{con:0jcon}
For any $j\geq 1$, we have
\[ s_{0,j}=\frac{(2j+5)+\sqrt{4j^2+16j+17}}{4(j+2)} .\]
If this conjecture holds, the lower bound is achieved by the $t$-blowups of $K_{2j+4}^{(j+2)*}$.
\end{conjecture}

Note that the lower bound is 
\[s_{0,j}\geq \frac{(2j+5)+\sqrt{4j^2+16j+17}}{4(j+2)} 
=1+\frac{1}{4j}-\frac{7}{16j^2} + \dotsm, 
\]
 while Theorem~\ref{thm:0jthm} gives
\[s_{0,j}\leq   \frac{1}{2}\left(1+\sqrt{\frac{j+2}{j+1}} \right)= 1+\frac{1}{4j}-\frac{5}{16j^2} + \dotsm. 
\]
Thus the lower bound and upper bound for $s_{0, j}$ are close for large $j$.

Determining $s_{i, j}$ for all pairs $(i, j)$ seems to be a hopeless problem. Let us at least mention some additional lower bounds which give rough asymptotics and some intuition. Nikiforov~\cite{Nik15} showed that for $n$ sufficiently large, there is a graph $G$ on $n$ vertices with $\lambda_{i+1}\ge \left(\frac{1}{2\sqrt{i}+\sqrt[3]{i+1}}-o(1)\right)n$ for a fixed sufficiently large $i$. If $n\ge 4^j$, then $\lambda_{n-j} \le 0$ (see \cite[Theorem 2.3]{Nik15}), so $s_{i, j} \ge \frac{1}{2\sqrt{i} + \sqrt[3]{i+1}}$, while if $i$ is fixed and $j$ is much larger than $i$, then by Theorem~\ref{thm:genijub}, we have
\[s_{i, j} \le \frac{1}{2}\sqrt{\frac{i+j+1}{i(j+1)}} =  \frac{1}{2\sqrt{i}}\left(1 + \frac{i}{2(j+1)} - \frac18\left(\frac{i}{j+1}\right)^2 + \ldots\right).\]
For small $i$, the $n$-vertex graphs with largest known $\lambda_{i+1}/n$~\cite{Nik15, Lin} likewise give reasonable lower bounds on $\lambda_{i+1} - \lambda_{n-j}$ when $j$ is large. Intuitively, this suggests that when $i$ is fixed and $j$ is much larger than $i$, maximizing the spread $s_{i, j}$ is similar to maximizing $\lambda_{i+1}/n$. Similar results hold if $j$ is fixed. 

 We now consider the case that $i$ and $j$ are relatively close to each other. Let $G^*$ be the graph of order $4\cdot 2^{m}$ achieving equality in Theorem \ref{thm:allk}, where $m$ is the positive integer such that $2^{m-1} < \max\{i, j+1\} \leq 2^m$. The existence of $G^*$ is guaranteed by the existence of a Hadamard matrix of order $2\cdot2^{m}$. For any $0 \leq i,j \leq 2^m,$ we have the following bounds:
\[
\frac{1}{2\sqrt{\max\{i, j+1\}}} <  \frac{1}{2\sqrt{2^{m-1}}} \leq s_{i,j} \leq \frac{1}{\sqrt{2\min\{i, j+1\}}}
\]
In particular, if $\min\{i, j+1\} \ge 2^{m-1}$, then $s_{i, j}$ is known to within a factor of $\frac{1}{\sqrt{2}}$. 

It would be interesting to determine $s_{i, j}$ exactly for more pairs $(i, j)$. We conclude with tables showing the best bounds on $s_{i, j}$ known for small values of $i$ and $j$. In Table~\ref{tab:bounds}, for each pair $(i, j)$ the best known lower bound is listed in the bottom left and the best known upper bound in the top right corner. The upper bounds are obtained from either Theorem~\ref{thm:genijub} or Theorem~\ref{thm:0jthm}, while the lower bounds are obtained from the blowups of specific graphs. For the pairs $(0, 0), (1, 0), (2, 1)$ and $(4, 3)$ we list the known exact values obtained from either \cite{BRTU2021+} or from Theorem~\ref{thm:allk}. In Table~\ref{tab:graphs}, for the pairs $(0, 0), (1, 0), (2, 1), (4, 3)$ we list the exact extremal graphs (or the graph's adjacency matrix) and for all other pairs we list the graphs which give the best known lower bounds. The graphs listed in Table~\ref{tab:graphs} which have not been previously defined in the paper were obtained by computer search; the sparse6 representation of each of these graphs is listed in Table~\ref{tab:sparse6}. 

\setlength{\tabcolsep}{2pt}
\begin{table}[ht]
    \centering
    \begin{tabular}{|C{0.8cm}||*{5}{C{2.8cm}|}}
        \hline
        \diagbox[innerleftsep=4pt, innerwidth = 20pt]{$i$}{$j$} & 0 & 1 & 2 & 3 & 4 \\
        \hline\hline
        0 & $2/\sqrt{3}$~\cite{BRTU2021+} & \mt{1.090}{1.112} & \mt{1.066}{1.077} & \mt{1.052}{1.059} & \mt{1.043}{1.048}\\
        \hline
        1 & $1/\sqrt{2}$ & \mt{0.600}{0.612} & \mt{0.571}{0.577} & \mt{0.556}{0.559} & \mt{0.545}{0.547}\\
        \hline
        2 & \mt{0.600}{0.612} & $1/2$ & \mt{0.441}{0.456} & \mt{0.415}{0.433} & \mt{0.404}{0.418}\\
        \hline
        3 & \mt{0.571}{0.577}  & \mt{0.441}{0.456} & \mt{0.404}{0.408} & \mt{0.368}{0.382} & \mt{0.341}{0.365}\\
        \hline
        4 & \mt{0.556}{0.559} & \mt{0.415}{0.433} & \mt{0.368}{0.382} & $\sqrt{2}/4$ & \mt{0.315}{0.335}\\
        \hline
    \end{tabular}
    \caption{Known lower and upper bounds for $s_{i,j}$.}
    \label{tab:bounds}
\end{table}
\newpage
\begin{table}[ht]
    \centering
    \begin{tabular}{|C{0.8cm}||*{5}{C{2.8cm}|}}
        \hline
        \diagbox[innerleftsep=4pt, innerwidth = 20pt]{$i$}{$j$} & 0 & 1 & 2 & 3 & 4 \\
        \hline\hline\xrowht{0.75cm}
        0\,\, & $K_{3}^{(2)*}$ \cite{BRTU2021+}  & $K_{6}^{(3)*}$ & $K_{8}^{(4)*}$ & $K_{10}^{(5)*}$ & $K_{12}^{(6)*}$\\
        \hline\xrowht{0.75cm}
        1\,\, & $P^*_4$ & $K_{3} \cup K_{2}^*$ & $K_{4} \cup K_{3}^*$ & $K_{5} \cup K_{4}^*$ & $K_{6} \cup K_{5}^*$\\
        \hline\xrowht{0.75cm}
        2\,\, & $K_{(3)*,2}$ & $Q_3^*$ & $G_1$ & $G_2$ & $G_3$ \\
        \hline\xrowht{0.75cm}
        3\,\, & $K_{(4)*,3}$ & $G_1^c$ & $G_4$ & $G_5$ & $G_6$ \\
        \hline\xrowht{0.75cm}
        4\,\, & $K_{(5)*,4}$ & $G_2^c$ & $G_5^c$ & $\frac{1}{2}(K \otimes H_8 + J_{16})$ & $G_7$ \\
        \hline
    \end{tabular}
    \caption{Best known constructions for $s_{i,j}$.}
    \label{tab:graphs}
\end{table}
\setlength{\tabcolsep}{6pt}
\begin{table}[h!]
    \centering
    \begin{tabular}{|c|c|}
        \hline
        Graph & sparse6 \\
        \hline\hline
        $G_1$ & \verb|:K_ES`s_QOqDL?G`f_C`SOAGXsoAOiCqEOhdJ| \\
        $G_2$ & \verb|:N_EC?aF?G`c_E?Qe_CXAecaPSQEPATQEPATTK`IdtK\ATkiWyCkYz|\\
        $G_3$ & \verb|:Oc?GgbaMGqOL?PbsIWyIDK\AXcIXATOAGXW@CKawAK\ATk_CXAiUq?PEMlbV^|\\
        $G_4$ & \verb|:FehIA_t_S|\\
        $G_5$ & \verb|Graph on 20 vertices|\tablefootnote{All graphs in Table \ref{tab:sparse6} can be viewed at \url{https://github.com/ghbrooks28/maximumSpectralGaps}.}\\
        $G_6$ & \verb|:K@GKPT?QXAecOhxBGWyG@CLC?bGSqTOAG`RhV|\\
        $G_7$ & \verb|:J`?S@oBG[aDeOpwbJCPsHaOhc^|\\
        \hline
    \end{tabular}
    \caption{Sparse6 representation for previously undefined graphs in Table \ref{tab:graphs}.}
    \label{tab:sparse6}
\end{table}

\end{document}

%% file: figures/closedP4.tikz
 \tikz[scale=2.5, vertex/.style={scale=0.5, circle, draw=black, fill=black},
wvertex/.style={scale=0.5, circle, draw=black, fill=white}]{
\node[vertex] (v1) at (0,0) {};
\node[wvertex](v2) at (1,0 ) {};
\node[wvertex] (v3) at (2,0) {};
\node[vertex] (v4) at (3,0) {};
\draw (v1)--(v2);
\draw (v2)--(v3);
\draw (v3)--(v4);
}

%% file: figures/closedCube.tikz
 \tikz[scale=2.5, vertex/.style={scale=0.5, circle, draw=black, fill=black},
wvertex/.style={scale=0.5, circle, draw=black, fill=white}]{
\node[vertex] (v1) at (0,0) {};
\node[wvertex](v2) at (1,0 ) {};
\node[vertex] (v3) at (1.4,0.4) {};
\node[wvertex] (v4) at (0.4,0.4) {};
\node[vertex] (v5) at (0,1) {};
\node[wvertex](v6) at (1,1 ) {};
\node[vertex] (v7) at (1.4,1.4) {};
\node[wvertex] (v8) at (0.4,1.4) {};
\draw (v1)--(v2);
\draw (v2)--(v3);
\draw (v3)--(v4);
\draw (v1)--(v4);
\draw (v2)--(v4);
\draw (v5)--(v6);
\draw (v6)--(v7);
\draw (v7)--(v8);
\draw (v6)--(v8);
\draw (v1)--(v5);
\draw (v2)--(v6);
\draw (v3)--(v7);
\draw (v4)--(v8);
\draw (v5)--(v8);
}

%% file: figures/conj1.tikz
\begin{tikzpicture}[scale=2.5, vertex/.style={scale=0.5, circle, draw=black, fill=black},
wvertex/.style={scale=0.5, circle, draw=black, fill=white}]

\def \number {5}
\def \radius {0.75cm}
\def \degree {360/\number}

\foreach \s in {1,...,\number}
{
    \node[wvertex] (\s) at ({\degree * (\s -1)+90}:\radius) {};
}
\foreach \s in {1,...,\number}{
    \foreach \t in {1,...,\number}{
        \ifnum \s < \t
            \draw (\s)--(\t);
        \fi
        }}

\begin{scope}[xshift=2cm]

\def \number {4}
\def \radius {0.75cm}
\def \degree {360/\number}

\foreach \s in {1,...,\number}
{
    \node[vertex] (\s) at ({\degree * (\s -1)+45}:\radius) {};
}
\foreach \s in {1,...,\number}{
    \foreach \t in {1,...,\number}{
        \ifnum \s < \t
            \draw (\s)--(\t);
        \fi
        }}
\end{scope}
\end{tikzpicture}

%% file: figures/conj2.tikz
\begin{tikzpicture}[scale=2.5, vertex/.style={scale=0.5, circle, draw=black, fill=black},
wvertex/.style={scale=0.5, circle, draw=black, fill=white}]

\foreach \s in {1,...,4}
{
    \node[wvertex] (\s) at (\s-1, 1) {};
}
\foreach \s in {5,...,9}
{
    \node[vertex] (\s) at (\s-5.5, 0) {};
}
\foreach \s in {1,...,4}{
    \foreach \t in {5,...,9}{
        \draw (\s)--(\t);
        }}
\end{tikzpicture}

%% file: figures/conj3.tikz
\begin{tikzpicture}[scale=2.5, vertex/.style={scale=0.5, circle, draw=black, fill=black},
wvertex/.style={scale=0.5, circle, draw=black, fill=white}]

\def \number {8}
\def \radius {0.75cm}
\def \degree {360/\number}

\foreach \s in {1,...,\number}
{
    \ifnum \s < 5
        \node[wvertex] (\s) at ({\degree * (\s -1)+22.5}:\radius) {};
    \else
        \node[vertex] (\s) at ({\degree * (\s -1)+22.5}:\radius) {};
    \fi
}
\foreach \s in {1,...,\number}{
    \foreach \t in {1,...,\number}{
        \ifnum \s < \t
            \draw (\s)--(\t);
        \fi
        }}
        
\end{tikzpicture}